\documentclass[11pt]{article}

\usepackage{epsfig}
\usepackage{graphicx}
\usepackage{amsbsy}
\usepackage{amsmath}
\usepackage{amsfonts}
\usepackage{amssymb}
\usepackage{textcomp}
\usepackage{hyperref}
\usepackage{aliascnt}

\newcommand{\mcm}[3]{\newcommand{#1}[#2]{{\ensuremath{#3}}}} 

\mcm{\tuple}{1}{\langle #1 \rangle}
\mcm{\name}{1}{\ulcorner #1 \urcorner}
\mcm{\Nbb}{0}{\mathbb{N}}
\mcm{\Zbb}{0}{\mathbb{Z}}
\mcm{\Rbb}{0}{\mathbb{R}}
\mcm{\Cbb}{0}{\mathbb{C}}
\mcm{\Qbb}{0}{\mathbb{Q}}
\mcm{\Bcal}{0}{\cal B}
\mcm{\Ccal}{0}{\cal C}
\mcm{\Dcal}{0}{\cal D}
\mcm{\Ecal}{0}{\cal E}
\mcm{\Fcal}{0}{\cal F}
\mcm{\Gcal}{0}{\cal G}
\mcm{\Hcal}{0}{\cal H}
\mcm{\Ical}{0}{\cal I}
\mcm{\Jcal}{0}{\cal J}
\mcm{\Kcal}{0}{\cal K}
\mcm{\Lcal}{0}{\cal L}
\mcm{\Mcal}{0}{\cal M}
\mcm{\Ncal}{0}{\cal N}
\mcm{\Ocal}{0}{{\cal O}}
\mcm{\Pcal}{0}{{\cal P}}
\mcm{\Qcal}{0}{{\cal Q}}
\mcm{\Rcal}{0}{{\cal R}}
\mcm{\Scal}{0}{{\cal S}}
\mcm{\Tcal}{0}{{\cal T}}
\mcm{\Ucal}{0}{{\cal U}}
\mcm{\Vcal}{0}{{\cal V}}
\mcm{\Xcal}{0}{{\cal X}}
\mcm{\Ycal}{0}{{\cal Y}}
\mcm{\Mfrak}{0}{\mathfrak M}

\mcm{\restric}{0}{\upharpoonright}
\mcm{\upset}{0}{\uparrow}
\mcm{\onto}{0}{\twoheadrightarrow}
\mcm{\smallNbb}{0}{{\small \mathbb{N}}}
\DeclareMathOperator{\preop}{op}
\mcm{\op}{0}{^{\preop}}

\newcommand{\se}{\subseteq}
{\begin{array}{c}
\setlength{\unitlength}{1em}}%
{\end{array}}

\usepackage{amsthm}

\newcommand{\theoremize}[2]{\newaliascnt{#1}{thm} \newtheorem{#1}[#1]{#2} \aliascntresetthe{#1}}

\theoremstyle{plain}
\newtheorem{thm}{Theorem}[section]
\theoremize{lem}{Lemma}
\theoremize{sublem}{Sublemma}
\theoremize{claim}{Claim}
\theoremize{rem}{Remark}
\theoremize{prop}{Proposition}
\theoremize{cor}{Corollary}
\theoremize{que}{Question}
\theoremize{oque}{Open Question}
\theoremize{con}{Conjecture}

\theoremstyle{definition}
\theoremize{dfn}{Definition}
\theoremize{eg}{Example}
\theoremize{exercise}{Exercise}
\theoremstyle{plain}

\usepackage{verbatim}
\usepackage{enumerate}
\usepackage[all]{xy}

\usepackage{subfig}

\title{\scshape Topological infinite gammoids, and a new Menger-type theorem for infinite graphs}

\author{Johannes Carmesin}

\newcommand{\sm}{\setminus}
\newcommand{\ct}{^\complement}

\begin{document}
 
\maketitle

\begin{abstract}
Answering a question of Diestel,
we develop a topological notion of gammoids in infinite graphs 
which, unlike traditional infinite gammoids, always define a matroid.

As our main tool, we prove for any infinite graph $G$ with vertex sets $A$ and $B$
that if every finite subset of $A$ is linked to $B$ by disjoint paths,
then the whole of $A$ can be linked to the closure of $B$ by disjoint paths or rays
in a natural topology on $G$ and its ends.

This latter theorem re-proves and strengthens the infinite Menger theorem of Aharoni and Berger
for `well-separated' sets $A$ and $B$. It also implies the topological Menger theorem 
of Diestel for locally finite graphs.
\end{abstract}

\section{Introduction}

Unlike finite gammoids, traditional infinite gammoids do not necessarily define a matroid. Diestel \cite{matroid_workshop} asked whether a suitable topological notion of infinite gammoid might mend this, so that gammoids always give rise to a matroid. We answer this in the positive by developing such a topological notion of infinite gammoid.
Our main tool is a new topological variant of Menger's theorem for infinite graphs, which is also interesting in its own right.

Given a directed graph $G$ with a set $B\se V(G)$ of vertices, 
the set $\Lcal(G,B)$ contains all vertex sets $I$ that can be linked by vertex-disjoint directed paths\footnote{In this paper, \emph{paths} are always finite.} to $B$. If $G$ is finite, $\Lcal(G,B)$ is the set of independent sets of a matroid, called the \emph{gammoid of $G$ with respect to $B$}. 
If $G$ is infinite, $\Lcal(G,B)$ does not always define a matroid \cite{AFM:gammoids}.

In 1968, Perfect \cite{perfect} looked at the question of when $\Lcal(G,B)$ is a matroid.
As usual at that time, she restricted her attention to matroids with every circuit finite,
now called \emph{finitary matroids}.
In \cite{matroid_axioms},  Bruhn et al found a more general notion of infinite matroids, which are closed under duality and need not be finitary.
Afzali, Law and M\"uller \cite{AFM:gammoids} study infinite gammoids in this more general setting and find conditions under which $\Lcal(G,B)$ is a matroid. 
In this paper, we introduce a topological notion of gammoids in infinite graphs that always define a matroid.

\vspace{0.3cm}

These gammoids can be defined formally without any reference to topology, as follows.
A ray $R$ in $G$ \emph{dominates $B$} if $G$ contains infinitely many vertex-disjoint directed paths from $R$ to $B$.
A vertex $v$ \emph{dominates $B$} if there are infinitely many directed paths from $v$ to $B$ that are vertex-disjoint except in $v$.
A path \emph{dominates $B$} if its last vertex dominates $B$.
A \emph{domination linkage} from $A$ to $B$ is a family of vertex-disjoint directed paths or rays
$(Q_a\mid a\in A)$ where $Q_a$ starts in $a$ and either ends in some vertex of $B$ or else
dominates $B$. 
A vertex set $I$ is in $\Lcal_T(G,B)$ if there is a domination linkage from $I$ to $B$.
We offer the following solution to Diestel's question:

\begin{thm}\label{is:finitarization_intro}
$\Lcal_T(G,B)$ is a finitary matroid.
\end{thm}

When $G$ is undirected\footnote{Formally, we consider those directed graphs $G$ obtained from an undirected graph by replacing each edge by two parallel edges directed both ways.}, \autoref{is:finitarization_intro} has the following topological interpretation. 
On $G$ and its ends consider the topology whose basic open sets are the components $C$ of $G\sm X$ where $X$ is a finite set of inner points of edges, together with the ends that have rays in $C$. The closure of $B\se V(G)$ consists of $B$, the vertices dominating $B$, and the ends $\omega$ whose rays $R\in \omega$ dominate $B$. Thus $I\in \Lcal_T(G,B)$ if and only if
the whole of $I$ can be linked to the closure of $B$ by vertex-disjoint paths or 
rays.\footnote{Instead of just taking paths and rays, one might want to take all `topological arcs'. However, this would result in a weaker theorem.}
We shall not need this topological interpretation:

\autoref{is:finitarization_intro} can be used to prove that under certain conditions the naive, non-topological, gammoid $\Lcal(G,B)$ is a matroid, too:

\begin{cor}\label{nearly_finitary_intro}
Let $G$ be a digraph with a set $B$ of vertices such that there are neither infinitely many vertex-disjoint rays dominating $B$ nor
 infinitely many vertices dominating $B$.
Then $\Lcal(G,B)$ is a matroid.
\end{cor}

\autoref{nearly_finitary_intro} does not follow from the existence criterion of Afzali, Law and M\"uller for non-topological gammoids. Also its converse is not true, see \autoref{Sec:gammo} for details.

\vspace{0.3 cm}

The main tool in our proof of \autoref{is:finitarization_intro} is a 
purely graph-theoretic Menger-type theorem, which seems to be interesting in its own right.
It is not difficult to show that if there is a domination linkage from $A$ to $B$, then there is a linkage from every finite subset of $A$ to $B$.
Our theorem says that the converse is also true:

\begin{thm}\label{main_thm_intro2}
\begin{enumerate}[\rm(i)]
 \item In any infinite digraph with vertex sets $A$ and $B$, there is a domination linkage from $A$ to $B$ if and only if
every finite subset of $A$ can be linked to $B$ by vertex-disjoint directed paths.
\item In any infinite undirected graph $G$, a set $A$ of vertices can be linked by disjoint paths and rays to the closure of another vertex set $B$ if and only if every finite subset can be linked to $B$ by vertex-disjoint paths.
\end{enumerate}
\end{thm}

We remark that the proof of \autoref{main_thm} is non-trivial and not merely a compactness result. Applying compactness, one would get a topological linkage from $A$ to the closure of $B$ by arbitrary topological arcs, not necessarily paths and rays. Our graph-theoretical version of \autoref{main_thm} is considerably stronger than this purely topological variant.

In \autoref{apps1} we study the relationship between \autoref{main_thm_intro2}
and existing Menger-type theorems for infinite graphs: the Aharoni-Berger theorem \cite{AharoniBerger}
and the topological Menger theorem for arbitrary infinite graphs. The latter was proved by Bruhn, Diestel and Stein \cite{eme}, extending an earlier result of Diestel \cite{countableEME} for countable graphs.
For infinite graphs with `well-separated' sets $A$ and $B$ (defined in \autoref{apps1}), \autoref{main_thm_intro2}
implies and strengthens the Aharoni-Berger theorem.
This in turn allows us to give a proof of the 
topological Menger theorem for locally finite graphs 
which, unlike the earlier proofs, does not rely on the 
(countable) Aharoni-Berger theorem (which was proved earlier by Aharoni \cite{aharoni87}).

The paper is organised as follows.
After a short preliminary section we prove in \autoref{proof}
the directed edge version of \autoref{main_thm_intro2}.
In \autoref{apps1} we sketch how this variant implies \autoref{main_thm_intro2},
and how \autoref{main_thm_intro2} implies the 
Aharoni-Berger theorem for `well-separated' sets $A$ and~$B$, and the topological Menger theorem for locally finite graphs. In \autoref{Sec:gammo} we summarise some basics about infinite matroids, and prove \autoref{is:finitarization_intro} and \autoref{nearly_finitary_intro}.

\section{Preliminaries}\label{prelims}

Throughout, notation and terminology for graphs are that of~\cite{DiestelBook10}.
In this paper, we will mainly be concerned with sets of edge-disjoint directed paths.
Thus, we abbreviate edge-disjoint by disjoint, edge-separator by separator and directed path by path.
Given a digraph $G$ and $A,B\se V(G)$, a \emph{linkage from $A$ to $B$} is a set of disjoint paths from the whole of $A$ to $B$.
We update the definitions of when a ray dominates $B$, a vertex dominates $B$, a path dominates $B$, and what a domination linkage is: these are the definitions made in the Introduction with ``vertex-disjoint'' replaced by ``edge-disjoint''.
The proof of the following theorem takes the whole of \autoref{proof}.

\begin{thm}\label{main_thm}
Let $G$ be a digraph and $b\in V(G)$, and $I\se V(G)-b$.
There is a domination linkage from $I$ to $\{b\}$ if and only if
every finite subset of $I$
has a linkage into $b$.
\end{thm}

We delay the proof that \autoref{main_thm} implies \autoref{main_thm_intro2} 
until \autoref{apps1}.
In a slight abuse of notion, we shall suppress the set brackets of 
$\{b\}$ and just talk about ``domination linkages from $I$ to $b$''.
One implication of \autoref{main_thm} is indeed easy:

\begin{lem}\label{easy_implication}
 If there is a domination linkage from $I$ to $b$, then
every finite subset $S$ of $I$
has a linkage into $b$.
\end{lem}

\begin{proof}
For $s\in S$, let $P_s$ be the path or ray from the domination linkage starting in $s$.
Suppose for a contradiction, there is no linkage from $S$ into $b$.
Then by Menger's theorem, there is a set $F$ of at most $|S|-1$ edges such that after its removal there is no directed path from $S$ to $b$.

Suppose for a contradiction that there is some $P_s$ not containing an edge of $F$. Then $P_s$ cannot end at $b$.
So $P_s$ dominates $b$, and thus there is some $P_s$-$b$-path avoiding $F$, contradicting the fact that $F$ was an edge-separator.
Thus each $P_s$ contains an edge of $F$. As the $P_s$ are edge-disjoint, $|F|\geq |S|$, which is the desired contradiction. 
\end{proof}

\section{Proof of \autoref{main_thm}}\label{proof}

The proof of \autoref{main_thm} takes the whole of this section.

\subsection{Exact graphs}

The core of the proof of \autoref{main_thm} is the special case where $G$ is exact (defined below).
In this subsection, we show that the special case of \autoref{main_thm} where $G$ is exact
implies the general theorem. More precisely, we prove that the \autoref{main_thm_exact} below
implies \autoref{main_thm}.

Given a vertex set $D$, an edge is \emph{$D$-crossing (or crossing for $D$)} if its starting vertex is in $D$ and the endvertex is outside. We abbreviate $V(G)\sm D$ by $D\ct$.
The \emph{order of $D$} is the number of $D$-crossing edges. 
The vertex set $D$ is \emph{exact} (for some set $I\se V(G)$ and $b\in V(G)$) if $b\notin D$ and the order of $D$ is finite and equal
to $|D\cap I|$. 
A graph is \emph{exact} (for $b$ and $I$) if for every $v\in V(G)-b$, there is an exact
set $D$ containing $v$.

\begin{lem}\label{main_thm_exact}
Let $G$ be an exact digraph and $b\in V(G)$.
Let $I\se V(G)-b$ such that every finite subset of $I$
has a linkage into $b$.
Then there is a domination linkage from $I$ to $b$.
\end{lem}

First we need some preparation.
 Let $G$ be a graph and let $b\in V(G)$.
Let $\Ical$ be the set of all sets $I\se V(G)-b$ such that every finite subset of $I$
has a linkage into $b$. The following is an easy consequence of Zorn's lemma.

\begin{rem}\label{fini_matroid}
Let $I\in \Ical$, and $X\se E(G)$, then there is $J\in \Ical$ maximal with
$I\se J\se X$.
\end{rem}

\begin{lem}\label{exact_property}
 Let $G$ be a directed graph, and let $I\se V(G)-b$ be maximal with the
property that every finite subset of $I$
has a linkage into $b$. Let $v\in (V(G)-b)\sm I$.
Then there is an exact $D$
containing $v$.
\qed
\end{lem}

\begin{proof}
By the maximality of $I$, there is a finite subset $I'$ of $I$ such that $I'+v$
cannot be linked to $b$. By Menger's theorem, there is a set $D$ of order at most $|I'|$ 
not containing $b$ but containing $I'+v\se D$. The order must be precisely $|I'|$
since $I'$ can be linked to $b$. Thus $D$ is exact, which completes the proof.
\end{proof}


\begin{proof}[Proof that \autoref{main_thm_exact} implies \autoref{main_thm}.]
By \autoref{easy_implication}, it suffices to prove the ``if''-implication. 
Let $G$, $b$, $I$ be as in \autoref{main_thm}.
We obtain the graph $H_1$ from $G$ by identifying $b$ with all vertices 
$v$ such that there are infinitely many 
edge-disjoint $v$-$b$-paths.
Note that in $H_1$ every vertex $v\neq b$ can be separated from $b$ by a finite separator.

It suffices to prove the theorem for $H_1$ since 
then the set of dominating paths and rays we get for $H_1$ extends to a set of
dominating paths and rays for $G$ by adding a singleton path for every vertex in
$I$ that is identified with $b$ in $H_1$.

We build an exact graph $H_2$ that has $H_1$ as a subgraph.
Let $v\in V(H_1)-b$. Let $k_v$ be the smallest order of some vertex set $D$
containing $D$ and not containing $b$. By construction of $H_1$, the number $k_v$ is finite.

We obtain $H_2$ from $H_1$ by for each $v\in V(H_1)-b$ adding $k_v$-many vertices
whose forward neighbourhood is that of $v$ and that do not have any incoming edges.
We shall refer to these newly added vertices for the vertex $v$ as the clones of $v$.

Now we extend $I$ to a maximal set $I_2\se V(H_2)-b$ such that every finite subset of $I_2$
has a linkage into $b$. This is possible by \autoref{fini_matroid}.

Next we show that $H_2$ is exact with respect to $I_2$.
Suppose for a contradiction that there is some $v\in I_2$
such that there is no exact $D$ containing $v$.
First we consider the case that $v\in V(H_1)$. Since $v$ together with all its clones cannot be linked to $b$, there is a clone $w$ of $v$ that is not in $I_2$.
Since $w\not\in I_2$, there must be some exact $D'$ containing $w$ by \autoref{exact_property}. If $v\in D'$, we are done, otherwise we get a contradiction since there is no linkage from $((I\cap D')+v)$ to $b$.
The case that $v\not\in V(H_1)$ is similar.

Having shown that $H_2$ is exact, we now use the assumption that 
\autoref{main_thm_exact} is true for $H_2$ and $I_2$: 
We get for each $v\in I$
some path or ray that dominates $b$ in $H_2$.
This path or ray also dominates in $H_1$ because a clone-vertex cannot be an
interior vertex of any directed path or ray. 
And it also dominates in $G$,  which completes the proof 
that \autoref{main_thm_exact} implies \autoref{main_thm}.
\end{proof}

\subsection{Exact vertex sets}

In this subsection, we prove some lemmas needed in the proof of \autoref{main_thm_exact}.

Until the end of the proof of \autoref{main_thm}, we shall fix a graph $G$
that is exact with respect to a fixed vertex $b$ and some set $I\se V(G)-b$.
We further assume that every finite subset of $I$
has a linkage into $b$. 
First we shall prove some lemmas about exact vertex sets.

\begin{lem}\label{exact1}
 Let $D$ be exact and let $P_1,\ldots P_n$ be a linkage from $I\cap D$ to $b$.
Then each $P_i$ contains precisely one $D$-crossing edge, and
each $D$-crossing edge is contained in one $P_i$.
\end{lem}

\begin{proof}
Clearly, each $P_i$ contains an $D$-crossing edge.
Since the $P_i$ are edge-disjoint no two of them contain the same crossing edge.
Since $D$ is exact, there are precisely $n$ $D$-crossing edges, and thus
there is precisely one on each $P_i$.
\end{proof}

\begin{lem}\label{paths_disjoint}
 Let $D,D'\se V(G)$ such that $D'\se D$, and $D'$
is exact.
Let $\Lcal$ be a linkage from $(I\cap D)$ to $b$.
If some $P\in\Lcal$ starts at a vertex in $D\sm D'$, then no vertex of $P$ lies in $D'$. 
\end{lem}

\begin{proof}
Since  $D'$ is exact, each $D'$-crossing edge lies on some path of $\Lcal$.
On the other hand $|I\cap D'|$ of the paths start in $D'$, and thus 
contain 
an $D'$-crossing edge. So $P$ cannot contain any $D'$-crossing edge. 
If $P$ meets $D'$, then it would meet $D'$ in a last vertex, and the edge
pointing away from this vertex would be an $D'$-crossing edge. Hence $P$
does not meet $D'$, which completes the proof.
\end{proof}

\begin{lem}\label{together}
 Let $D$ and $D'$ be exact.
\begin{enumerate}[(I)]
 \item Then $D\cup D'$ is exact. \label{corner1_is_exact}
\item Then $D\cap D'$ is exact.   \label{corner2_is_exact}
\item Then there does not exist an edge from  \label{cor:no_edge}
$D\sm D'$ to $D'\sm D$.
\end{enumerate}
\end{lem}

\begin{proof}
Let $\Lcal$ be a linkage from $I\cap (D\cup D')$ to $b$. 
For $X\se D\cup D'$, let
$\Lcal(X)$ denote the set of those paths in $\Lcal$ that 
have their starting vertex in $X$.
For $X\se V(G)$, let 
$\Ccal(X)$ denote the set of $X$-crossing edges.
It is immediate that.
\begin{equation}\label{eq1}
 |\Lcal(D\cap D')|+|\Lcal(D\cup D')|=|\Lcal(D)|+|\Lcal(D')|
\end{equation}
Since $D$ and $D'$ are exact, (\ref{eq1}) gives the following:
\begin{equation}\label{eq2}
 |\Lcal(D\cap D')|+|\Lcal(D\cup D')|=|\Ccal(D)|+|\Ccal(D')|
\end{equation}
Next, we prove the following.
\begin{equation}\label{eq1.5}
 |\Ccal(D\cap D')|+|\Ccal(D\cup D')|\leq|\Ccal(D)|+|\Ccal(D')|
\end{equation}

Each edge in both $\Ccal(D\cap D')$ and $\Ccal(D\cup D')$ points from $D\cap D'$
to $D\ct\cap {D'}\ct$, and hence is in both $\Ccal(D)$ and $\Ccal(D')$.
Each edge in $\Ccal(D\cap D')$ is in either $\Ccal(D)$ or $\Ccal(D')$.
Similarly, each edge in $\Ccal(D\cup D')$ is in either $\Ccal(D)$ or $\Ccal(D')$.
This proves inequation (\ref{eq1.5}). Note that if we have equality, we cannot have 
an edge from $D\sm D'$ to $D'\sm D$.

In order to prove (\ref{corner1_is_exact}) and (\ref{corner2_is_exact}), it suffices to show that $|\Lcal(D\cup D')|=|\Ccal(D\cup D')|$ and that 
$|\Lcal(D\cap D')|=|\Ccal(D\cap D')|$.
Since $\Lcal(D\cup D')$ and $\Lcal(D\cap D')$ are sets of edge-disjoint paths,
that each contain at least one crossing edge,
it must be that $|\Lcal(D\cup D')|\leq|\Ccal(D\cup D')|$ and that 
$|\Lcal(D\cap D')|\leq|\Ccal(D\cap D')|$.

By equations (\ref{eq1.5}) and (\ref{eq2}), we get that
\begin{equation}\label{eq3}
|\Ccal(D\cap D')|+|\Ccal(D\cup D')|\leq  |\Lcal(D\cap D')|+|\Lcal(D\cup D')| 
\end{equation}

Combining this with the two inequalities before, we must have that
$|\Lcal(D\cup D')|=|\Ccal(D\cup D')|$ and 
$|\Lcal(D\cap D')|=|\Ccal(D\cap D')|$, which proves (I) and (II).

Now it must be that we must have equality in (\ref{eq1.5}).
So there cannot be an edge from
$D\sm D'$ to $D'\sm D$, which proves (III). This completes the proof.

\end{proof}

\begin{lem}\label{push_sep_forward1}
 Let $G$ be an exact graph, and $F$ be a finite set of vertices.
Then there is an exact $D$ with $F-b\se D$.
\end{lem}

\begin{proof}
For each $v\in F-b$, there is an exact $D_v$ containing $v$ by exactness of $G$. Then $\bigcup_{v\in F-b} D_v$ is exact, which can easily be proved by induction over $|F-b|$, using (\ref{corner1_is_exact}) of \autoref{together} in the induction step. 
\end{proof}

Let $D$ be exact and let $\Lcal$ be a linkage from $D\cap
I$ to 
$b$. Then $D'$ is called a \emph{forwarder of $D$ with respect
to $\Lcal $}
if $D'$ is exact and $\bigcup \Lcal-b\se D'$ and $D\se D'$.

\begin{lem}\label{push_sep_forward}
 Let $G$ be an exact graph. Then each exact $D$
has an forwarder with respect to each linkage $\Lcal$ from
 some subset of 
$D\cap I$ to 
$b$.
\end{lem}

\begin{proof}
 Apply \autoref{push_sep_forward1} to the set of all vertices in $\bigcup \Lcal$
to get a $D'$ with all those vertices in $D'+b$.
The desired forwarder is then $D\cup D'$.
\end{proof}

The \emph{hull} $\hat D$ of a vertex set $D$ consists of those vertices that are separated by the $D$-crossing edges from $b$.
Note that $D\se\hat D$ and that $D\ct$ consists of those vertices $v$ such that there is a $v$-$b$-path all of whose internal vertices are outside $D$.
Since every vertex on such a path is in
${\hat D}\ct$, the hull of any hull $\hat D$ is
$\hat D$ itself.

We say that two vertex sets $D$ and $D'$ are \emph{equivalent} if 
they have the same hull. This clearly defines an equivalence relation,
which we shall call $\sim$.  Note that $D\sim D'$ if and only if $D$ and $D'$ have the same crossing edges.

\begin{rem}\label{equi}
 Let $F$, $F'$, $\tilde F$ and $\tilde F'$ be
exact with $\tilde F\sim F$ and $\tilde F'\sim F'$.
Then $F\cup F' \sim \tilde F\cup \tilde F'$.
\end{rem}

\begin{proof}
Clearly, the set of $(F\cup F')$-crossing edges is equal to the set
of 
$(\tilde F\cup \tilde F')$-crossing edges, which
gives the desired result.
\end{proof}

\begin{lem}\label{exactness_inherited}
For any exact $D$, the hull $\hat D$ is exact.
\end{lem}

\begin{proof}
It suffices to show that $I\cap \hat D=I\cap D$.
Since $\hat D\supseteq D$, clearly $I\cap \hat D \supseteq I\cap D$.
In order to prove the other inclusion, suppose for a contradiction that there
is some $v\in I\cap (\hat D\sm D)$. 
Since $(I\cap D)+v$ is finite, there is some 
linkage $\Lcal$ from $((I\cap D)+v)$ to $b$. 
Let $P$ be the path from that linkage that starts in $v$.
By \autoref{paths_disjoint}, the path $P$ avoids $D$.
So $P$ witnesses that $v\notin \hat D$.
This is a contradiction, thus $I\cap \hat D=I\cap D$.
\end{proof}

\subsection{Good functions}

We define what a good function is and prove that the existence of a good function in every exact graph implies \autoref{main_thm_exact}.

First, we fix some notation.
Let $\Ecal$ be the set of exact vertex sets $D$,
and let $\bar \Lcal$ be the set of linkages from finite subsets of $I$
to
$b$. For a vertex set $D$, the set $N(D)$ consists of $D$ 
together with all endvertices of $D$-crossing edges.

For $v\in I$ and some linkage $\Lcal$, let $Q_v(\Lcal)$ denote the path in $\Lcal$
starting from $v$. 
For every exact $D$, 
the edges of $Q_v(\Lcal)$ contained in $G[N(D)]$ 
are the edges of some initial path of $Q_v(\Lcal)$.
We call this initial path
$P_v(D;\Lcal)$. We follow the convention that $P_v(D;\Lcal)$ is empty if $v\not\in D$.

A function
$f:\Ecal\to \bar \Lcal$ is \emph{good} if it satisfies the following:
\begin{enumerate}[(i)]
 \item $f(F)$ is a linkage from $I\cap F$ to $b$.

\item If $v\in I$ and $F,F'\in dom(f)$ with $F'\se F$,
then $P_v(F';f(F'))=P_v(F';f(F))$.

\item If $\bigcup P_v(F;f(F))$ is a ray, then it dominates $b$. Here the
union ranges over all exact $F$.
\end{enumerate}

Before proving that there is a good function,
we first show how to deduce \autoref{main_thm} from that.
Let us abbreviate 
$P_v(D;f(D))$ by $P_v(D;f)$.
If it is clear by the context, which function $f$ we mean, we 
even just write $P_v(D)$.

\begin{lem}\label{compatible_partial}\label{compatible}
Let $f:\Ecal\to \bar \Lcal$ be a partial function 
satisfying (\ref{r1}) and (\ref{r2}).
Further assume that for any two exact $F$ and $F'$ with $F'\se F$
and $F\in dom(f)$, also $F'\in dom(f)$.

Let $v\in I$, and let $D,D'\in dom(f)$ be exact with $v\in D\cap D'$.
Then $P_v(D) \se P_v(D')$, or $P_v(D') \se
P_v(D)$.
\end{lem}

\begin{proof}
$D\cap D'$ is exact by \autoref{together} and in the domain of $f$. Since $f$ satisfies (\ref{r2}),
we get that $P_v(D\cap D')$ is a subpath of both $P_v(D)$ and $P_v(D')$.
Let $e$ be the last edge of $P_v(D\cap D')$, and $x$ be its endpoint in $D\ct\cup {D'}\ct$. 

Now we distinguish three cases.
If $x\in D\ct\cap {D'}\ct$, the edge $e$ is crossing for both $D$ and $D'$, and thus is the last edge of both $P_v(D)$ and $P_v(D')$.
So $P_v(D)=P_v(D')$, so the lemma is true in this case.

If $x\in D\ct\cap D'$, then $e$ is the last edge of $P_v(D)$. So
$P_v(D)=P_v(D\cap D')\se P_v(D')$, so the lemma is true in this case.

The case $x\in {D'}\ct\cap D$ is similar to the last case.
This completes the proof.
\end{proof}

For the remainder of this subsection, let us fix a good function $f$.
The last Lemma motivates the following definition. For $v\in I$, let
$P_v$ be the union of all the paths $P_v(D)$ over all exact $D$ containing $v$. By the last Lemma $P_v$ is either a path or a ray.

\begin{lem}\label{P_v_disjoint}
 If $P_v$ and $P_w$ share an edge, then $v=w$.
\end{lem}

\begin{proof}
Let $e$ be an edge in both $P_v$ and $P_w$.
Let $D_v$ be exact with $e\in P_v(D_v)$.
Similarly, let $D_w$ be exact with $e\in P_w(D_w)$.

By (\ref{corner1_is_exact}) of \autoref{together}, we get that $D_v\cup D_w$ is
exact.
Since $f$ is good, we have that $P_v(D_v\cup D_w)$ includes $P_v(D_v)$, and 
that $P_w(D_v\cup D_w)$ includes $P_w(D_w)$. Since $P_v(D_v\cup D_w)$ and
$P_w(D_v\cup D_w)$
 share the edge $e$, we must have that $v=w$, which completes the proof.
\end{proof}

\begin{lem}\label{P_v_is_path}
 If $P_v$ is a path, then it ends at $b$.
\end{lem}

\begin{proof}
Suppose for a contradiction that $P_v$ does not end at $b$.
Then $P_v$ does not contain $b$.

Then by \autoref{push_sep_forward1}, there is an exact $D$
with $P_v\se D$. 
Then $P_v(D)$ contains some $D$-crossing edge whose endvertex does not lie
on $P_v$, which gives a
contradiction to the construction of $P_v$.
\end{proof}

The following lemma tells us that to prove \autoref{main_thm},
it remains to show that every exact graph has a good function.

\begin{lem}\label{exact_case}
Let $G$ be an exact digraph that has a good function. Let $b\in V(G)$.
Let $I\se V(G)-b$ such that every finite subset of $I$
has a linkage into $b$.
Then there is a domination linkage from $I$ to $b$.
\end{lem}

\begin{proof}
Each $P_v$ dominates $b$: If $P_v$ is a path, this is shown  in 
\autoref{P_v_is_path}. If $P_v$ is a ray, this follows from the fact that $f$
is good.
By \autoref{P_v_disjoint} all the $P_v$ are edge-disjoint, which completes
the proof.
\end{proof}

\subsection{Intermezzo: The countable case}

The purpose of this subsection is to prove that there is a good function under the assumption that $G$ is countable. This case is easier than the general case and some of the ideas can already be seen in this special case. However, in the general case we do not rely on the countable case. At the end of this subsection, we tell why this proof does not extend to the general case. We think that this helps to get a better understanding of the general case.

\begin{lem}\label{countable_nested_part}
 Let $G$ be an exact graph with $V=\{v_0=b,v_1,v_2,\ldots\}$ countable. Then there is a sequence of exact hulls $D_n$ and linkages $\Lcal_n$ from $I\cap D_n$ to $b$ satisfying the following.
\begin{enumerate}
 \item $D_{n}\se D_{n+1}$;
\item $\{v_1,\ldots ,v_n\}\se D_n$;
\item $P_v(D_n;\Lcal_n)=P_v(D_n;\Lcal_{n+1})$ for any $v\in I$;
\item $D_{n+1}$ is a forwarder of $D_n$ with respect to $\Lcal_n$.
\end{enumerate}

\end{lem}

\begin{proof}
Assume that for all $i\leq n$, we already constructed exact hulls $D_i$ and linkages $\Lcal_i$ satisfying 1-4.

Next, we define $D_{n+1}$. By \autoref{push_sep_forward1}, there is an exact 
 $F_n$ containing $v_{n+1}$.
By \autoref{together}, $D_n\cup F_n$ is exact.
Let $D_{n+1}'$ be a forwarder of $D_n\cup F_n$ with respect to the linkage $\Lcal_n$, which exists by \autoref{push_sep_forward}. 
Let $D_{n+1}$ be the hull of $D_{n+1}'$, which is exact by \autoref{exactness_inherited}.

It remains to construct $\Lcal_{n+1}$ so as to make 3 true.
Let $\Lcal$ be some linkage from $I\cap D_{n+1}$ to $b$.
By \autoref{exact1}, for each $D_n$-crossing edge $e$ there is precisely one $P_e\in \Lcal_n$ that contains $e$, and precisely one $Q_e\in \Lcal$ that contains $e$.
Let $R_e=P_eeQ_eb$. Since $P_ee\se D_n+e$ and $eQ_eb\se D_n\ct+e$, the $R_e$ are edge-disjoint.
For $\Lcal_{n+1}$ we pick the set of the $R_e$ together with all $Q\in \Lcal$ that do not contain any $D_n$-crossing edge. Clearly $\Lcal_{n+1}$ is an 
linkage from $I\cap D_{n+1}$ to $b$. And 3 is true by construction. This completes the construction.
\end{proof}

\begin{lem}\label{countable_case}
Every countable exact graph $G$ has a good function $f$.
\end{lem}

\begin{proof}
Let $D_n$ and $\Lcal_n$ as in \autoref{countable_nested_part}.
We let $f(D_n)=\Lcal_n$.
Next, we define $f$ at all other exact $D$. Since there are only finitely many $D$-crossing edges, there is a large number $m$ such that all these crossing edges are in $N(D_m)$. Then $D\se D_m$ as
for each $v\notin D_m$ there is a $v$-$b$-path included avoiding $D_m$.
Now we let $f(D)$ consist of those paths in $f(D_m)$ that start in $D$.
We remark that this definition does not depend on the choice of $m$.

Having defined $f$, it remains to check that it is good: clearly it satisfies (i) and (ii), and it just remains to verify (iii). 
So assume that for some $v\in I$, the union $R=\bigcup_{F\in \Ecal} P_v(F;f(F))$ is a ray.
Then $R=\bigcup_{n\in \Nbb} P_v(D_n;\Lcal_n)$.
Let $e_v^n$ be the unique $D_n$-crossing edge on $Q_v(\Lcal_n)$.
Since $D_{n+1}$ is a forwarder of  $D_{n}$, the path 
$R_v^n=e_v^nQ_v(\Lcal_n)$ is contained in  $D_n+b$ and avoids $D_{n+1}$.
Thus the paths  $R_v^n$ are edge-disjoint and witness that $R$ dominates $b$.
So $f$ is good, which completes the proof.
\end{proof}

\begin{rem}\label{why_count}
 Our proof above heavily relies on the fact that we can find a nested set of exact vertex sets $D_n$ indexed with the natural numbers that exhaust the graph (compare 2 in \autoref{countable_nested_part}). However if we can find such a nested set, then $I$ must be countable since each $D_n$ contains only finitely many vertices of $I$.
Thus this proof does not extend to the general case.
\end{rem}

\subsection{Infinite sequences of exact vertex sets}

The purpose of this subsection is to prove some lemmas that help proving that there is a good function in every exact graph. These lemmas are about infinite
sequences of exact vertex sets. 

\begin{lem}\label{no_inf_seq}
There does not exist a sequence $(D_n|n\in \Nbb)$  with $D_n\subsetneq D_{n+1}$ 
of exact hulls that all have bounded order.
\end{lem}

\begin{proof}
Suppose for a contradiction that there is a such sequence $(D_n|n\in
\Nbb)$. By taking a subsequence if necessary, we may assume that all $D_n$ have the same order.
Since any two $D_n$ are exact and have the same order,
we must have $I\cap D_1=I\cap D_n$ for every $n$.
Let $\Lcal$ be some linkage from $(D_1\cap I)$ to $b$. Any $P\in \Lcal$
contains a unique $D_n$-crossing edge for every $n$ by \autoref{exact1}.
Since $D_n\se
D_{n+1}$, there is a large number $n_P$
such that for all $n\geq n_P$ it is the same crossing edge.
Let $m$ be the  maximum of the numbers $n_P$ over all $P\in \Lcal$.
Then for all $n\geq m$, the $D_n$ have the same crossing edges
and thus are equivalent. This is a contradiction,
completing the proof.
\end{proof}

\begin{lem}\label{nested_limit}
 Let $D$ be exact and let $\Xcal$ 
be a nonempty set of exact $D'\se D$ that is closed under $\sim$ and taking unions.
Then there is some $D''\in \Xcal$ including all $D'\in\Xcal$.
\end{lem}

\begin{proof}
Suppose for a contradiction that there 
is no such $D''\in \Xcal$.
We shall construct an infinite sequence $(D_n|n\in \Nbb)$ as in \autoref{no_inf_seq}.

Let $D_1\in \Xcal$ be arbitrary. 
Since $\Xcal$ is $\sim$-closed, we may assume that $D_1$ is its
hull. Now assume that $D_n$ is already constructed. 
By assumption, there is $D_n'\in \Xcal$ with $D_n' \not \se D_n$.
Let $D_n''=D_n\cup D_n'$. 
Let $D_{n+1}$ be the hull of $D_n''$.
Then 
 $D_{n+1}\in \Xcal$, and $D_n\subsetneq D_{n+1}$.
This completes the construction of the 
infinite sequence $(D_n|n\in \Nbb)$, which contradicts \autoref{no_inf_seq} and hence completes the
proof. 
\end{proof}

\subsection{Existence of good functions}

The purpose of this subsection is to prove that every exact graph has a good
function,
which implies \autoref{main_thm} by \autoref{exact_case}.
Next we shall define when a partial function is good for the following reason.
In order to construct a good function $f$ defined on the whole of $\Ecal$ we
shall construct an ordinal indexed family
of good partial functions $f_\alpha$ such that if $\alpha>\beta$, then the
domain of $f_\alpha$ includes that of $f_\beta$ and agrees with $f_\beta$ on the
domain of $f_\beta$.
Eventually some $f_\alpha$ will be defined on the whole of $\Ecal$ and will
be the desired good function.

The domain of a partial function $f$ is denoted by  $dom(f)$.
A partial function $f:\Ecal\to ¸\bar \Lcal$ is \emph{good} if it satisfies the
following:
\begin{enumerate}[(i)]
 \item $f(F)$ is a linkage from $I\cap F$ to $b$.\label{r1}
\item If $v\in I$ and $F,F'\in dom(f)$ with $F'\se F$,\label{r2}
then $P_v(F';f(F'))=P_v(F';f(F))$.

\item If $\bigcup P_v(F;f(F))$ is a ray, then it dominates $b$. Here the
union ranges over all $F\in dom(f)$. \label{r5}

\item Let $F$ and $F'$ be exact with $F'\se F$.
If $F$ is in the domain of $f$, then so is $F'$.\label{r3}

\item If $F$ and $F'$ are in the domain of $f$, then so is 
$F\cup F'$. \label{r4}
\item $dom(f)$ is
closed under $\sim$. \label{r6}
\end{enumerate}

Note that if $F,F'\in dom(f)$, then so is
$F\cap F'$ by (\ref{r3}). Note that each good partial function defined on
the whole of $\Ecal$ is a good function.

\begin{lem}\label{sim-closure}
Let $f$ be a partial function with domain $X$ that satisfies (\ref{r1})-(\ref{r4}).
Then there is a good partial function $\hat f$ whose domain is the {$\sim$-closure $\hat X$}
 of $X$
such that $\hat f\restric_X=f$.
\end{lem}

\begin{proof}
For each $F\in \hat X$, there is some $\tilde F\in X$ such that $F\sim \tilde F$.
We let $\hat f(F)=f(\tilde F)$.
Clearly, $\hat f$ satisfies (\ref{r1}), (\ref{r5}) and (\ref{r6}).
Since $f$ satisfies (\ref{r4}) and by \autoref{equi}, $\hat f$ satisfies (\ref{r4}).

To see that $\hat f$ satisfies (\ref{r2}), let 
 $v\in I$ and $F,F'\in dom(\hat f)$ with $F'\se F$.
Then $P_v(F';\hat f(F'))=P_v(F';\hat f(F))$ as $P_v(\tilde F';f(\tilde F')=P_v(\tilde F';f(\tilde F)$.

To see that $\hat f$ satisfies (\ref{r3}), let 
$F$ and $F'$ be exact with $F'\se F$ and $F\in \hat X$. 
Then $F'\cap \tilde F$ is exact, and since $f$ satisfies 
(\ref{r3}), it must be in $X$. 
Since $F'\cap \tilde F$ and $F'$ have the same crossing edges,
 they are equivalent. So $F'\in \hat X $. So $\hat f$ satisfies 
(\ref{r3}).
This completes the proof.
\end{proof}

For $S \se \Ecal$, let $S(\ref{r3})\se \Ecal$ denote the smallest set including $S$
that satisfies (\ref{r3}). Similarly, let $S(\ref{r4})\se \Ecal$ denote the smallest set
including $S$ that
satisfies (\ref{r4}).

\begin{lem}\label{X_char}
$[S(\ref{r3})](\ref{r4})=[S(\ref{r4})](\ref{r3})$ for any set $S$.

 \noindent In particular, $[S(\ref{r3})](\ref{r4})$ is the smallest set included in
$\Ecal$ 
satisfying (\ref{r3}) and (\ref{r4}).
\end{lem}
\begin{proof}
First let $D\in [S(\ref{r3})](\ref{r4})$. Then there are
$F_1,F_2\in S(\ref{r3})$ such that $D=F_1\cup F_2$. Then there are $F_1',F_2'\in S$ such that
$F_1\se F_1'$ and $F_2\se F_2'$.
Then $F_1'\cup F_2'\in S(\ref{r4})$ by (\ref{corner1_is_exact}) of \autoref{together}.
Since $F_1\cup F_2\se F_1'\cup F_2'$, we deduce that 
$D\in [S(\ref{r4})](\ref{r3})$.
So $[S(\ref{r3})](\ref{r4})\se [S(\ref{r4})](\ref{r3})$.

The proof that $[S(\ref{r4})](\ref{r3})\se [S(\ref{r3})](\ref{r4})$ is similar:
Now let $D\in [S(\ref{r4})](\ref{r3})$. 
Then there is $D'\in S(\ref{r4})$ with $D\se D'$.
Then there are
$F_1,F_2\in S$ such that $D'=F_1\cup F_2$. Then $F_i\cap D\in S(\ref{r3})$
for $i=1,2$ by (\ref{corner2_is_exact}) of \autoref{together}. 
Since $D\se F_1\cup F_2$, we deduce that $D\in [S(\ref{r3})](\ref{r4})$.
This completes the proof.
\end{proof}

Let $X \se \Ecal$, and $D$ be exact. 
Then $X[D]$ denotes 
the smallest set including $X+D$ 
that satisfies (\ref{r3}) and (\ref{r4}).

\subsubsection{Extending good partial functions}

The aim of this subsubsection is to prove the following lemma 
that helps us building a good function in that it allows us to extend a good partial function a little bit.

\begin{lem}\label{odd_successor_step}
 Let $f$ be a good partial function, and let $D$ be exact.
Then there is a good partial function $g$ whose domain consists of the
$\sim$-closure of $dom(f)[D]$,
and that agrees with $f$ at each point in $dom(f)$. 
\end{lem}

If $D\in dom(f)$, then we just take $g=f$.
So we may assume that $D\not\in dom(f)$.
Before we define $g$, we define auxiliary functions $g_1$, $g_2$ and $g_3$
with domains $X_1$, $X_2$ and $X_3$, respectively, such that $dom(f)\se X_1\se X_2\se X_3\se dom(g)$, and 
$g$ will be defined such that $g\restric_{X_1}=g_1$, $g\restric_{X_2}=g_2$, and 
$g\restric_{X_3}=g_3$.
We let $X_1=dom(f)+D$.

For all $D'\in dom(f)$, we let $g_1(D')=f(D')$.
Next we define $g_1(D)$.
Since $I\cap D$ is finite, there is some linkage from $I\cap D$ into $b$.
Let $P_1,P_2,\ldots P_n$ be such a linkage. 
By \autoref{nested_limit},
there is some $D''\in  dom(f) $ with $D''\se D$ such that $D'\se D''$ for all $D'\in
dom(f) $ with $D'\se D$.
Since $D''$ is exact, each $D''$-crossing edge lies on one of the
$P_i$ by \autoref{exact1}.

We define $g_1(D)$ as follows.
If no $D''$-crossing edge lies on $P_i$, then we put $P_i$ into
$g_1(D)$. If some $D''$-crossing edge, say $e_i$, lies on $P_i$,
we take the path $Q_i$ from the linkage $f(D'')$ that contains $e_i$,
and put
the path $Q_ie_iP_i$ into $g_1(D)$. This completes the definition of 
$g_1(D)$, and so of $g_1$. 

\begin{sublem}
 $g_1$ satisfies (\ref{r2}).
\end{sublem}

\begin{proof}
Let $F, F'\in X_1$ with $F'\se F$. If $F$ is not $D$, then 
$P_v(F';f(F'))=P_v(F';f(F))$ since $f$ satisfies (\ref{r2}) and (\ref{r3}).

So we may assume that $F=D$. Then $F'\se D''\se D$.
So $P_v(F';f(F'))=P_v(F';f(D''))=P_v(F';f(D))$, which completes the proof. 
\end{proof}

Having defined $X_1$ and $g_1$, we now define $X_2$ and $g_2$.
We let $X_2=X_1(\ref{r3})$. For each $F\in X_2$ there is some $F'\in X_1$
such that $F \se F'$. We let $g_2(F)$ to consists of those paths from 
$g_1(F')$ that start in $F$.
By construction, $g_2$ satisfies (\ref{r1}) and (\ref{r3}).
By \autoref{together}(\ref{cor:no_edge}), $P_v(F,g_2(F))=P_v(F,g_1(F'))$ for all $v\in I$.
Thus $g_2$ satisfies (\ref{r2}) as $g_1$ does.

Having defined $g_2$, we now define $g_3$.
We let $X_3$ be $X_2(\ref{r4})$, which is equal to $dom(f)[D]$. We let $P_v=P_v(F;g_2)\cup P_v(F';g_2)$. By
\autoref{compatible_partial}, it must be that  $P_v=P_v(F;g_2)$ or $P_v= P_v(F';g_2)$.
Since $g_2$ satisfies (\ref{r2}), no vertex of $P_v(F;g_2)$ that is not on $P_v(F';g_2)$
can be in $F\cap F'$.
By (\ref{cor:no_edge}) of \autoref{together}, it must be that every vertex of $P_v(F;g_2)$ that is not on $P_v(F';g_2)$ is in $F\sm F'$. Hence the $P_v$ are edge-disjoint.

By (\ref{cor:no_edge}) of \autoref{together}, each $P_v$ contains some $(F\cup F')$-crossing edge $e_v$. Let $\Lcal$ be some linkage from $I\cap(F\cup F')$ to $b$. Let $Q_v$ be the path in $\Lcal$
that contains $e_v$. 
We define $g_3(F\cup F')$ to consist of the paths $P_ve_vQ_v$.
Clearly, $g_3(F\cup F')$ is a linkage from $I\cap(F\cup F')$ to $b$.

By \autoref{X_char}, $g_3$ satisfies not only (\ref{r4}) but also (\ref{r3}).

\begin{sublem}\label{sublem2}
$g_3$ satisfies (\ref{r2}).
\end{sublem}

\begin{proof}
Let $v\in I$ and $F,F'\in dom(f)$ with $F'\se F$.
Our aim is to prove that $P_v(F';g_3(F'))=P_v(F';g_3(F))$.
In the definition of $g_3$ at $F$, we have picked $F_1$ and $F_2$ in $X_2$ such that $F=F_1\cup F_2$ in order to define $g_3(F)$.
Similarly, we have picked $F_1'$ and $F_2'$ in $X_2$ such that 
with $F'=F_1'\cup F_2'$ to define $g_3(F')$.
It suffices to show that $P_v(X_{ij};g_3(F'))=P_v(X_{ij};g_3(F))$
where $X_{ij}=F_j'\cap F_i$ and $(i,j)\in \{1,2\}\times \{1,2\}$. 

By the definition of $g_3$, we get the following two equations. 
\begin{equation}\label{g_3_one}
 P_v(F';g_3(F'))= P_v(F_1';g_2)\cup P_v(F_2';g_2)
\end{equation}
\begin{equation}\label{g_3_two}
 P_v(F';g_3(F))= P_v(F'\cap F_1;g_2(F_1))\cup P_v(F'\cap F_2;g_2(F_2))
\end{equation}
Since $g_2$ satisfies (\ref{r2}), 
these two equations give the desired result when restricted to $X_{ij}$.
\end{proof}

\begin{sublem}\label{sublem3}
$g_3$ satisfies (\ref{r5}).
\end{sublem}

\begin{proof}
For each $v\in I$, we compare the sets $\bigcup P_v(F;f(F))$
where first the union ranges over all $F\in dom(f)$ and second it ranges over all $F\in dom(g_3)$.
The second set is a superset of the first and all its additional elements are in
$P_v(D,g_3)$, which is finite. In particular, if the second set is a ray, then so is the first set by \autoref{compatible_partial}. In this case, the first set dominates $b$ since $f$ satisfies (\ref{r5}), so the second set also dominates $b$.  
This completes the proof.
\end{proof}

\noindent Having defined $g_3$, we let
$g=\hat g_3$ as in \autoref{sim-closure}.
Since $g_3$ satisfies {(\ref{r1}) -(\ref{r4})}, $g$ is good by \autoref{sim-closure}. 
This completes the proof of \autoref{odd_successor_step}.

\subsubsection{Construction of a good function}

In this subsubsection, we construct a good function in every exact graph, which is the last step in the proof of \autoref{main_thm}.
Each ordinal $\alpha$ has a unique representation $\alpha=\beta+n$ where $\beta$
is the largest limit ordinal smaller than $\alpha$, and $n$ is a natural number.
We say that $\alpha$ is \emph{odd} if $n$ is odd. Otherwise it is \emph{even}.

\begin{lem}\label{exists_good_f}
Let $G$ be an exact graph.
Then there is a good function $f$ defined on the whole of $\Ecal$.
\end{lem}

\begin{proof}
In order to construct $f$ we shall construct an ordinal indexed family
of good partial functions $f_\alpha$ such that if $\alpha>\beta$, then the
domain of $f_\alpha$ includes that of $f_\beta$ and agrees with $f_\beta$ on the
domain of $f_\beta$.
Eventually some $f_\alpha$ will be defined on the whole of $\Ecal$ and will be
the desired good function.

Assume that $f_\beta$ is already defined for all $\beta<\alpha$.
First we consider the case that $\alpha=\beta+1$ is a successor ordinal. If
$f_\beta$ is defined on the whole of $\Ecal$, we stop.
Otherwise, we shall find some exact $F_\alpha$.
Then we let $f_\alpha$ be the partial function $g$ given to us
from  \autoref{odd_successor_step} applied to
 $f_\beta$ and $F_\alpha$. 

How we find $F_\alpha$ depends on whether $\alpha$ is an odd or an
even successor ordinal. 
If $\alpha$ is odd, then we pick some $D\in \Ecal\sm dom(f_\beta)$, and let
$F_\alpha=D$.

If $\alpha$ is even, say $\alpha=\delta+2n$, where $\delta$ is the
largest limit ordinal less than $\alpha$, then for $F_\alpha$ we pick the 
forwarder of $F_{\delta+n}$ with respect to the
linkage $f(F_{\delta+n})$, which exists by
\autoref{push_sep_forward}.

\vspace{0.3 cm}

Having considered the case where $\alpha$ is a successor ordinal, we now
consider the case where $\alpha$ is a limit ordinal.
For the domain of $f_\alpha$ we take the union of the domains of all $f_\beta$
with $\beta<\alpha$, and we let $f_\alpha(D)=f_\beta(D)$ for some $\beta$
where this is defined. It is clear that $f_\alpha$ satisfies (\ref{r1}),(\ref{r2}),(\ref{r3}),(\ref{r4}),(\ref{r6}),
so it remains to show 
that $f$ satisfies (\ref{r5}).
So let $v\in I$ such that $R=\bigcup
P_v(D;f_\alpha)$ is a ray. Here the
union ranges over all $D$ in the domain of $f_\alpha$.

Let $\Ocal$ be the set of ordinals $\beta<\alpha$ such that 
there is some $D\in dom(f_\beta)$ with $v\in D$. $\Ocal$ is nonempty, so it must contain
a smallest ordinal $\epsilon$. Note that $\epsilon$ is a successor ordinal.
Let $\epsilon^-$ be such that $\epsilon=\epsilon^-+1$.

We shall prove that $v\in F_\epsilon$. Suppose not for a contradiction,
then $v\not\in D$ for all $D\in dom(f_{\epsilon^-})+F_\epsilon$.
So $v\not\in D$ for all $D\in [dom(f_{\epsilon^-})+F_\epsilon](\ref{r3})$,
and hence also $v\not\in D$ for all $D\in dom(f_{\epsilon^-})[F_\epsilon]$ by \autoref{X_char}.
By \autoref{exactness_inherited} and since $v\in I$, also 
$v\not\in D$ for all $D$ in the $\sim$-closure of  $dom(f_{\epsilon^-})[F_\epsilon]$. This contradicts the choice of $\epsilon$.
Hence $v\in F_\epsilon$.
Let $x$ be the unique $F_\epsilon$-crossing edge contained in 
$P_v(F_\epsilon;f_\alpha)$.

We have a representation $\epsilon=\delta+k$ where $\delta$ is the largest limit
ordinal less than $\epsilon$.
By construction, the $F_{\epsilon(l)}$ with 
$\epsilon(l)=\delta+2^l\cdot k$ are nested with
each other. To
prove that $R$ dominates $b$,
it will suffice just to investigate the $
F_{\epsilon(l)}$.

The paths $P_v(l)=Q_v(f_\alpha(F_{\epsilon(l)}))$ are contained in $F_{\epsilon(l+1)}+b$.
By \autoref{exact1}, there is a unique $F_{\epsilon(l)}$-crossing edge $a_l$ on $P_v(l)$.
The paths $a_lP_v(l)$ meet $F_{\epsilon(l)}$ only in their starting vertex.
Thus the paths $a_lP_v(l)$ are edge disjoint.
Since $a_l$ is on $P_v(F_{\epsilon(l)};f_\alpha)$, it is on $R$.
Hence the paths $a_lP_v(l)$ witness that $R$ dominates $b$.
So $f_\alpha$ is good.

There must be some successor step $\alpha=\beta+1$ at which we stop.
Then $f_\beta$ is a good function defined on the whole of $\Ecal$.
This completes the proof.
\end{proof}

\begin{proof}[Proof of \autoref{main_thm}]
Recall that the easy implication is already proved in \autoref{easy_implication}.
For the other implication, combine \autoref{exists_good_f} with \autoref{exact_case} to get a proof of \autoref{main_thm_exact}. Then remember that \autoref{main_thm_exact} implies \autoref{main_thm}.
\end{proof}

\section{Graph-theoretic applications of \autoref{main_thm}}\label{apps1}

In this section, we show how \autoref{main_thm} implies \autoref{main_thm_intro2} and 
how \autoref{main_thm_intro2}  implies the 
Aharoni-Berger theorem for `well-separated' sets $A$ and~$B$, and the topological Menger theorem for locally finite graphs.

\subsection{Variants of \autoref{main_thm}}\label{other_versions}

In this subsection, we explain how \autoref{main_thm} implies \autoref{main_thm_intro2}.
\autoref{main_thm} is equivalent to the following.
\begin{thm}\label{main_thm_intro1}
There is a domination linkage from $A$ to $B$ if and only if
every finite subset of $A$ can be linked to $B$.
\end{thm}
Menger's theorem comes in four different versions: the directed edge version,
the undirected edge version, the directed vertex version and the undirected vertex version.
Depending on the version, we have different notions of path, separator and disjointness.
Taking these different notions instead, we know in each of these 4 versions what it means that a ray dominates $B$, a vertex dominates $B$, a path dominates $B$, and what a domination linkage is, and what a linkage is.

The purpose of this subsection is to explain how \autoref{main_thm_intro1} implies its undirected-edge-version, directed-vertex-version and undirected-vertex-version. 
These versions are like \autoref{main_thm_intro1} but with
the appropriate notions of domination linkage and linkage.
The proof is done in the same way how one shows that the directed-edge-version of Menger's theorem for finite graphs implies all the other versions.

Starting with the sketch, one first shows that the directed-edge-version implies the directed-vertex-version for every graph $G$.
For this one considers the auxiliary digraph $H$ of $G$ with $V(H)=V(G)\times\{in, out\}$.
The edges of $H$ are of two types: For each $v\in V(G)$, we add an edge pointing from 
$(v,in)$ to $(v,out)$. For each edge of $H$ pointing from $v$ to $w$, we add an edge pointing from $(v,out)$ to $(w,in)$. Then the directed-vertex-version for $G$ is equivalent to the directed-edge-version for $H$.

Next one shows that the directed-vertex-version implies the undirected-vertex-version for every graph $G$. For this, one considers the directed graph $H$ obtained from $G$ by replacing each edge by two edges in parallel pointing in different directions. 
 Then the undirected-vertex-version for $G$ is equivalent to the directed-vertex-version for $H$.

Finally, one shows that the undirected-vertex-version implies the undirected-edge-version for every graph $G$. For this, one considers the line graph $H$ of $G$. 
Then the undirected-edge-version for $G$ is equivalent to the undirected-vertex-version for $H$.

It is clear that 
the directed vertex version of \autoref{main_thm_intro1} is just \autoref{main_thm_intro2}(i). We call domination linkages in the undirected vertex version \emph{vertex-domination linkages}. Similarly, we define \emph{vertex-linkages}.
The undirected vertex version of \autoref{main_thm_intro1} is the following.
\begin{cor}\label{main_thm_undirected_vertex}
Let $G$ be a graph and $A,B\se V(G)$.
There is a vertex-domination linkage from $A$ to $B$ if and only if
every finite subset of $A$
has a vertex-linkage into $B$.\qed
\end{cor}

\autoref{main_thm_undirected_vertex} is a reformulation of \autoref{main_thm_intro2}(ii).

\subsection{Well-separatedness}\label{sec:technical}

In this subsection, we prove \autoref{main_thm_intro_cor} below, which is used to deduce the 
Aharoni-Berger theorem for `well-separated' sets $A$ and~$B$, and the topological Menger theorem for locally finite graphs.
 
A pair $(A,B)$ of vertex sets is \emph{well-separated} if every vertex or end can be separated from one of $A$ or $B$ by removing finitely many vertices.

\begin{cor}(undirected vertex version)\label{main_thm_intro_cor}
Let $(A,B)$ be a well-separated pair of vertex sets.
Then there is a vertex-linkage from the whole of $A$ to $B$ if and only if every finite subset of $A$ has a vertex-linkage to $B$.
\end{cor}

Our next aim is to deduce \autoref{main_thm_intro_cor} from \autoref{main_thm_undirected_vertex}.
First we need some lemmas.
For this, we fix a graph $G$ and a well-separated pair $(A,B)$ of vertex sets.
Let $(P_a|a\in A)$ be a vertex-domination linkage from $A$ to $B$.
Let $\omega$ be an end that cannot be separated from $B$ by removing finitely many vertices. Let $A_\omega$ be the set of those $a\in A$ such that $P_a$ is a ray and belongs to $\omega$.

\begin{lem}\label{to_linkage}
There is a vertex-linkage $(Q_a|a\in A_\omega)$ from $A_\omega$ to $B$ such that 
$Q_a$ and $P_x$ are vertex-disjoint for all $a\in A_\omega$ and all $x\in A\sm A_\omega$.
\end{lem}

\begin{proof}
For a finite vertex set $S$, we denote by $C(S,\omega)$ the component of $G\sm S$ that contains $\omega$.

As $(A,B)$ is well-separated, 
there is a finite set of vertices $S$ that separates $\omega$ from $A$.
So the set $Z$ of those $a\in A$ such that $P_a$ meets $C(S,\omega)\cup S$ is finite.
As $A_\omega\se Z$, the set $A_\omega$ must be finite.
Furthermore there is a finite set $T$ such that $C(T,\omega)$ meets precisely those $P_a$
with $a\in A_\omega$. For $a\in A_\omega$, let $t_a$ be first vertex on $P_a$
such that $t_aP_a$ is contained in $C(T,\omega)$, which exist as these $P_a$ are eventually contained in $C(T,\omega)$. The $(t_aP_a|a\in A_\omega)$ form a vertex-domination linkage from $(t_a|a\in A_\omega)$ to $B$ in $G'$, where $G'$ is obtained from $G[C(T,\omega)]$ by deleting all edges on the paths $P_at_a$ with $a\in A_\omega$. By the easy implication of \autoref{main_thm_undirected_vertex} applied to $G'$, we get a vertex-linkage $(K_a|a\in A_\omega)$
from $(t_a|a\in A_\omega)$ to $B$. 
Each walk $P_at_aK_a$ includes a path $Q_a$ from $a$ to $B$.
From the construction, it is clear, that the $Q_a$ form a vertex-linkage from $A_\omega$ to $B$
and that 
$Q_a$ and $P_x$ are vertex-disjoint for all $a\in A_\omega$ and all $x\in A\sm A_\omega$.
\end{proof}

\begin{lem}\label{make_linkage_only_fin_path}
There is a vertex-domination linkage $(R_a|a\in A)$ from $A$ to $B$ such that each $R_a$ is a  path.
\end{lem}

\begin{proof}
We shall construct $(R_a|a\in A)$ by transfinite recursion. First we well-order the set $\Omega$ of ends: $\Omega=\{\omega_\alpha|\alpha\in \kappa\}$ for $\kappa=|\Omega|$.
At each step $\beta$ we have a current set of vertex-disjoint $A$-$B$-paths $\Qcal_\beta$.
The set  $A_\beta$ of  start vertices of paths in $\Qcal_\beta$ consists of those $a\in A$ such that $P_a$ is a ray and belongs to some end $\omega_\alpha$ with $\alpha<\beta$.
We shall also ensure that $\Rcal_\beta=\Qcal_\beta\cup \{P_a|a\not\in A_\beta\}$ is a vertex-domination linkage from $A$ to $B$. 

If $\beta$ is a limit ordinal, we just set $\Qcal_\beta=\bigcup_{\alpha< \beta}\Qcal_\alpha $. It is immediate that $\Qcal_\beta$ has the desired property assuming that the 
$\Qcal_\alpha $ with $\alpha<\beta$ have the property.
If $\beta=\alpha+1$ is a successor ordinal, we apply \autoref{to_linkage}
to the vertex-domination linkage $\Rcal_\alpha$. 
Then we let $\Qcal_{\alpha+1}= \Qcal_{\alpha}\cup \{Q_a|a\in A_{\omega_{\alpha+1}}\}$.
It is clear from that lemma that $\Qcal_{\alpha+1}$ has the desired property.

This completes the recursive construction. 
It is clear that $\Rcal_\kappa=\Qcal_\kappa\cup \{P_a|a\not\in A_\kappa\}$ is the desired vertex-domination linkage.
\end{proof}

\begin{proof}[Proof that \autoref{main_thm_undirected_vertex} implies \autoref{main_thm_intro_cor}.]
Let $(A,B)$ be well-separated such that from every finite subset of $A$ there is a vertex-linkage to $B$.
By \autoref{main_thm_undirected_vertex}, there is a vertex-domination linkage $(P_a|a\in A)$ from the whole of $A$ to $B$. 
By \autoref{make_linkage_only_fin_path}, we may assume that each $P_a$ is a path.
However, $(P_a|a\in A)$ may still contain a path $P_u$ that does not end in $B$.
Then $P_u$ has to contain a vertex $\omega$ that cannot be separated from $B$ by removing finitely many vertices. 
An argument as in the proof of \autoref{to_linkage}, shows that there is 
a path $Q_u$ from $u$ to some vertex in $B$ such that 
$(P_a|a\in A-u)$ together with $Q_u$ is a vertex-domination linkage from $A$ to $B$. Similar as in the proof of \autoref{make_linkage_only_fin_path}, we can now apply transfinite induction to replace each $P_u$ one by one by such a path $Q_u$. 
The final vertex-domination linkage is then a vertex-linkage, which completes the proof.
\end{proof}

\subsection{Existing Menger-type theorems}

In this subsection, we show how \autoref{main_thm_intro_cor} implies the 
Aharoni-Berger theorem for `well-separated' sets $A$ and~$B$, and the topological Menger theorem for locally finite graphs.

The Aharoni-Berger theorem \cite{AharoniBerger} says that for every graph $G$ with vertex sets $A$ and $B$, there is a set of vertex-disjoint $A$-$B$-paths together with an $A$-$B$-separator consisting of precisely one vertex from each of these paths. 

At first glance, it might seem that the Aharoni-Berger theorem does not tell under which conditions there is a linkage from $A$ to $B$ - but actually it does. To explain this, we need a definition.
A \emph{wave} is a set of vertex-disjoint paths from a subset of $A$ to 
some $A$-$B$-separator $C$. 
It is not difficult to show that the Aharoni-Berger theorem is equivalent to the following:
The whole of $A$ can be linked to $B$ if and only if 
for every wave there is a linkage from $A$ to its separator set $C$.
Thus \autoref{main_thm_intro_cor} implies the Aharoni-Berger theorem for well-separated sets $A$ and $B$. We remark that neither \autoref{main_thm_intro_cor} nor \autoref{main_thm} follows from the Aharoni-Berger theorem.

Using this implication, we get the first proof of the topological Menger-Theorem
 of Diestel \cite{countableEME} for locally finite graphs that does not rely on the Aharoni-Berger theorem. 
Indeed, the argument of Diestel only relies on the Aharoni-Berger theorem for vertex sets $A$ and $B$ that have disjoint closure in $|G|$, which is equivalent to being well-separated 
if $G$ is locally finite.

\section{Infinite gammoids}\label{Sec:gammo}

In this section, we use \autoref{main_thm_intro2} to prove \autoref{is:finitarization_intro} and \autoref{nearly_finitary_intro}.
Throughout, notation and terminology for matroids are that of~\cite{Oxley,matroid_axioms}.
$M$ always denotes a matroid and $E(M)$ and $\Ical(M)$ denote its ground set and its sets of independent sets, respectively. 

Recall that the set $\Ical(M)$ is required to satisfy the following \emph{independence axioms}~\cite{matroid_axioms}:
\begin{itemize}
	\item[(I1)] $\emptyset\in \Ical(M)$.
	\item[(I2)] $\Ical(M)$ is closed under taking subsets.	
	\item[(I3)] Whenever $I,I'\in \Ical(M)$ with $I'$ maximal and $I$ not maximal, there 	exists an $x\in I'\setminus I$ such that $I+x\in \Ical(M)$.
	\item[(IM)] Whenever $I\subseteq X\subseteq E$ and $I\in\Ical(M)$, the set $\{I'\in\Ical(M)\mid I\subseteq I'\subseteq X\}$ has a maximal element.
\end{itemize}

An \emph{$\Ical$-circuit} is a set minimal with the property that it is not in $\Ical$.
The following is true in any matroid.
\begin{itemize}
		\item[(+)] For any two finite $\Ical$-circuits $o_1$ and $o_2$ and any $x\in o_1\cap o_2$,
there is some $\Ical$-circuit included in $(o_1\cup o_2)-x$.
	\end{itemize}

Given $\Ical\se \Pcal(E)$, its \emph {finitarization} $\Ical^{fin}$
consists of those sets $J$ all of whose finite subsets are in $\Ical$. Usually, it is made a requirement that $\Ical$ is the set of independent sets of a matroid \cite{union1}. 
Then $\Ical^{fin}$ is the set of independent sets of a finitary matroid, called $M^{fin}$ \cite{union1}. We shall need the following slight strengthening of this fact.

\begin{lem}\label{+lem}
 If $\Ical$ satisfies (I1), (I2) and (+), then $\Ical^{fin}$ is the set of independent sets of a finitary matroid.
\end{lem}

\begin{proof}
 Clearly $\Ical^{fin}$ satisfies (I1) and (I2), and it satisfies (IM) by Zorn's Lemma.
Thus it remains to check (I3). So let $I,I'\in \Ical^{fin}$ with $I'$ maximal and $I$ not maximal.
So there is some $y\notin I$ with $I+y\in \Ical$. 
We may assume that $y\notin I'$ since otherwise we are done.
Thus there is some finite $\Ical$-circuit $o$ with $y\in o\se I'+y$.
Suppose for a contradiction that for each $x\in o\sm (I+y)$, there is some finite $\Ical$-circuit $o_x$ with $x\in o_x\se I+x$.
Applying (+) successively to $o$ and the $o_x$, we obtain a finite $\Ical$-circuit $o'$ included in $I+y$, which contradicts the assumption that $I+y\in \Ical^{fin}$.
Thus there is some $x\in o\sm (I+y)$ such that $I+x\in \Ical^{fin}$, which completes the proof.
\end{proof}

We shall also need the following slight variation of $(I3)$.
\begin{itemize}
		\item[(*)] For all $I, J\in\Ical$ and all $y\in I\setminus J$ with $J+y\notin\Ical$ there exists $x\in J\setminus I$ such that $(J + y) - x\in\Ical$.
	\end{itemize}

A matroid $N$ is \emph{nearly finitary} if 
for every base $B$ of $N$ there is a base $B'$ of $N^{fin}$ such that $B\se B'$
and $|B'\sm B|$ is finite. 
It is not difficult to show that $N$ is nearly finitary if and only if 
for every base $B'$ of $N^{fin}$ there is a base $B$ of $N$ such that $B\se B'$
and $|B'\sm B|$ is finite. 
The proof of Lemma 4.15 in \cite{union1} actually proves the following strengthening of itself.

\begin{lem}\label{nf}
 Let $M=(E,\Jcal)$ be a matroid with ground set $E$.
Let $\Ical\se \Jcal$ satisfying $(I1)$, $(I2)$ $(I3)$, $(*)$ such that for any $J\in \Jcal$ there is some
$I\in \Ical$ such that $|J\sm I|$ is finite. 
Then $N=(\Ical,E)$ is a matroid.
\end{lem}

In the special case where $M$ is finitary, $N$ is nearly finitary.

Next, we shall summarise the results from  \cite{AFM:gammoids} that are relevant to this paper.

\begin{lem}[Afzali, Law, M\"uller {\cite[Lemma 2.2]{AFM:gammoids}}]\label{AFM:I3}
For any digraph $G$ and $B\se V(G)$, the system $\Lcal(G,B)$ satisfies (I3).
\end{lem}

\begin{lem}[Afzali, Law, M\"uller {\cite[Lemma 2.7]{AFM:gammoids}}]\label{I4}
For any digraph $G$ and $B\se V(G)$, the system $\Lcal(G,B)$ satisfies $(^*)$.
\end{lem}

Let $B_{AC}=\{b_0,b_1,\ldots\}$. Let $V_{AC}=B_{AC}\cup V^1\cup V^2$, where $V^i=\{v_0^i,v_1^i,\ldots\}$.
The digraph $G_{AC}$ has vertex set $V_{AC}$ and three types of edges:
For $j\in \Nbb$ it has an edge from $v_j^1$ to $b_j$.
For each $j\in \Nbb$, it has two edges, both start at $v_{j}^2$, and end at $v_{j}^1$ and $v_{j+1}^1$.
The pair $(G_{AC},B_{AC})$ is called an \emph{alternating comb (AC)}.
A subdivision of AC is drawn in \autoref{fig:AC}.
Formally, a \emph{subdivision of AC} is a pair $(H_{AC},B_{AC})$ where $H_{AC}$ is obtained from $G_{AC}$ by replacing each directed edge $xy$ by a directed path from $x$ to $y$ that is internally disjoint from all other such paths.
Here edges from $V_2$ to $V_1$ are not allowed to be replaced by a trivial path\footnote{A \emph{trivial path} consists of a single vertex only.} but 
the edges $v_j^1b_j$ are allowed to be replaced by a trivial path. 
A pair $(G,B)$ \emph{has a subdivision of AC} if there is a subgraph $H_{AC}$ of $G$ and
$B_{AC}\se B\cap V(H_{AC})$ such that 
$(H_{AC},B_{AC})$ is isomorphic to a subdivision of $AC$.

\begin{thm}[Afzali, Law, M\"uller {\cite[Theorem 2.6]{AFM:gammoids}}]\label{AFM-thm}
Let $G$ be a digraph and $B\se V(G)$ such that $(G,B)$ has no a subdivision of AC. Then $\Lcal(G,B)$ is a matroid.
\end{thm}

For the remainder of this section, let $G$ denote a digraph and $B\se V(G)$.
In the following, we shall explore for which digraphs $G$ and sets $B$ the system $\Lcal_T(G,B)$ is the set of independent sets of a matroid, and how 
$\Lcal_T(G,B)$ relates to $\Lcal(G,B)$.
If $G$ is finite $\Lcal(G,B)$ is a matroid and thus satisfies (+). 
The latter easily extends to infinite graphs $G$. 
\begin{lem}\label{+}
 $\Lcal(G,B)$ satisfies (+).
\end{lem}
\begin{proof}[Sketch of the proof.]
Given two finite $\Lcal(G,B)$-circuits $o_1$ and $o_2$ intersecting in some vertex $x$, there are separations $(A_i,B_i)$ with $o_i\se A_i$ and $B\se B_i$ of order at most $|o_i|-1$.
Then with a lemma like \autoref{together}, one shows that either
$(A_1\cup A_2,B_1\cap B_2)$ or $(A_1\cap A_2,B_1\cup B_2)$
separates some $\Lcal(G,B)$-circuit $o\se (o_1\cup o_2)-x$ from $B$.
\end{proof}

Using \autoref{main_thm_intro2}, we can prove the following slight extension of \autoref{is:finitarization_intro}.

\begin{cor}\label{is:finitarization}
$\Lcal_T(G,B)=\Lcal(G,B)^{fin}$ for any digraph $G$ and $B\se V(G)$. 

Moreover, $\Lcal_T(G,B)$ is a finitary matroid.
\end{cor}

\begin{proof}
By \autoref{main_thm_intro2}, $\Lcal_T(G,B)$ consists of those sets $I$ all of whose finite subsets can be linked to $B$ by vertex-disjoint directed paths, and thus $\Lcal_T(G,B)=\Lcal(G,B)^{fin}$. 
As $\Lcal(G,B)$ satisfies (I1), (I2) and (+), $\Lcal_T(G,B)$ is a finitary matroid
by \autoref{+lem}.
\end{proof}

Next we prove the following slight strengthening of  \autoref{nearly_finitary_intro} from the Introduction. Below we shall refer to the definition of dominating as defined in the Introduction.
\begin{cor}\label{nearly_finitary}
Let $G$ be a digraph with a set $B$ of vertices.
Then $\Lcal(G,B)$ is a nearly finitary matroid if and only if there are neither infinitely many vertex-disjoint rays dominating $B$ nor
 infinitely many vertices dominating $B$.

\end{cor}

\begin{proof}
Clearly, if $\Lcal(G,B)$ is nearly finitary, there are neither infinitely many vertex-disjoint rays dominating $B$ nor
 infinitely many vertices dominating $B$.
Conversely, assume that there are neither infinitely many vertex-disjoint rays dominating $B$ nor infinitely many vertices dominating $B$.
$\Lcal(G,B)$ clearly satisfies (I1) and (I2), and it satisfies (I3) and $(^*)$
by \autoref{AFM:I3} and \autoref{I4}. 
Let $J\in \Lcal_T(G,B)$. By \autoref{main_thm_intro2},
we get for each $v\in J$  a ray or path $P_v$ starting at $v$ such that all these $P_v$ are vertex-disjoint. Moreover each such $P_v$ either ends in $B$ or is a ray dominating $B$ or its last vertex dominates $B$. 
Let $I$ be the set of those $v$ such that $P_v$ ends in $B$.
By assumption $J\sm I$ is finite. 
So by \autoref{is:finitarization}, we can apply \autoref{nf} with $\Jcal=\Lcal_T(G,B)$ to deduce that $\Lcal(G,B)$ satisfies (IM), and thus is a nearly-finitary matroid.
\end{proof}

A natural question that comes up is to ask how \autoref{AFM-thm} and
\autoref{nearly_finitary} relate to each other.
In \cite{AFM:gammoids}, Afzali, Law and M\"uller 
construct a pair $(G,B)$ without AC such that $\Lcal(G,B)$
is not nearly finitary. They also do it in a way to make $\Lcal(G,B)$ 3-connected. Thus \autoref{nearly_finitary} does not imply \autoref{AFM-thm}.

To see that \autoref{AFM-thm} does not imply \autoref{nearly_finitary}, 
let $G$ be the $3$ by $\Zbb$ grid, formally: $V(G)=\{1,2 ,3\}\times \Zbb$, see \autoref{fig:AC}. 
\begin{figure} [htpb]   
	\begin{center}
		\includegraphics[width=13cm]{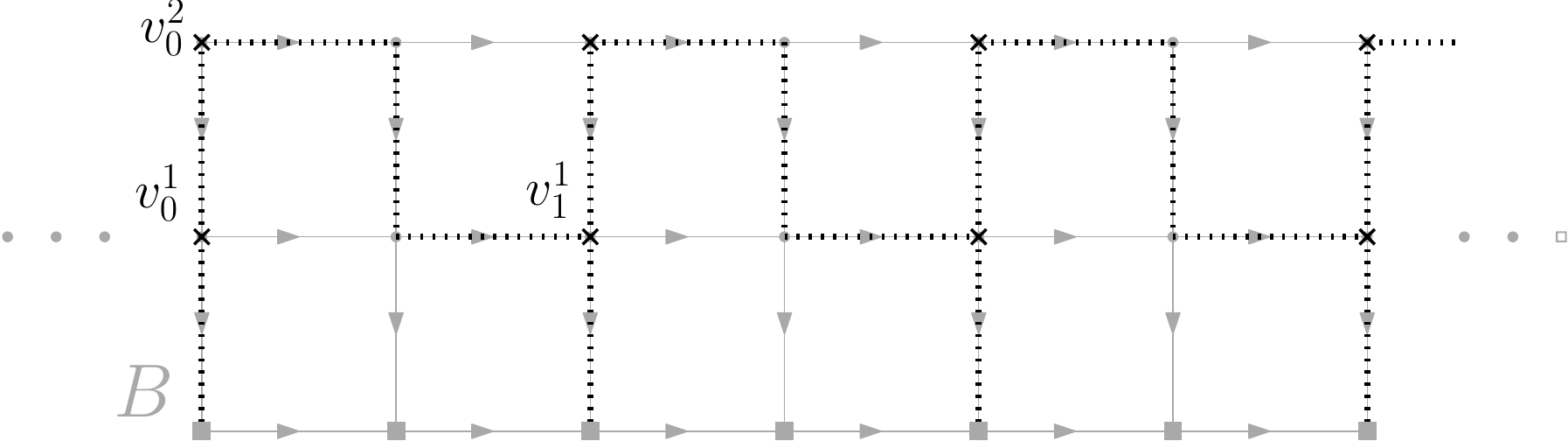}
		\caption{The graph $G$ is depicted in gray.  The vertices of $B$ are squares. $(G,B)$ has a subdivision of AC. One is indicated in this figure: The vertices of $V^1$ and $V^2$ are black crosses 
		and the subdivided edges are drawn dotted.}
		\label{fig:AC}
	\end{center}
\end{figure}
In $G$, there is a directed edge from $(x,y)$ to $(x',y')$ if and only if either $x=x'$ and $y'=y+1$ or
$y=y'$ and $x'=x+1$.
Let: $B=\{3\}\times \Zbb$.
Then it is easy to see that no vertex of $G$ dominates $B$ and there are not infinitely many  vertex-disjoint rays dominating $B$. However $(G,B)$ has a subdivision of AC, which is indicated in \autoref{fig:AC}.
Thus, there arises the question if there is a nontrivial common generalization of 
\autoref{nearly_finitary} and \autoref{AFM-thm}.

During this whole section, we have only considered the directed-vertex-version. Of course, similar results are true if we consider the undirected-vertex-version, the directed-edge-version or the undirected-edge-version instead.

\section{Acknowledgements}

I am grateful to Hadi Afzali, Hiu Fai Law and Malte M\"uller for explaining me many things about infinite gammoids.

\bibliographystyle{plain}
\bibliography{literatur}

\end{document}